\documentclass{article}

\usepackage[english]{babel}

\usepackage[letterpaper,top=2cm,bottom=2cm,left=3cm,right=3cm,marginparwidth=1.75cm]{geometry}

\usepackage{amssymb,amsmath,amsthm}
\usepackage{graphicx}
\usepackage[colorlinks=true, allcolors=blue]{hyperref}

\newtheorem{theorem}{Theorem}[section]
\newtheorem{lemma}[theorem]{Lemma}
\newtheorem{conjecture}[theorem]{Conjecture}

\newtheorem{claim}[theorem]{Claim}

\newcommand{\norm}[1]{\left\lVert#1\right\rVert}
\newcommand{\floor}[1]{\left\lfloor#1\right\rfloor}

\title{Packing edge disjoint cliques in graphs}
\author{    J\'ozsef Balogh\footnote{
    Department of Mathematics, University of Illinois at Urbana--Champaign, Urbana, IL, USA. Research supported in part by NSF grants RTG DMS-1937241, FRG DMS-2152488, the Arnold O.~Beckman Research Award (UIUC Campus Research Board RB 24012), and the Simons Fellowship. E-mail: jobal@illinois.edu.} \and     Michael C. Wigal\footnote{Department of Mathematics, University of Illinois at Urbana--Champaign, Urbana, IL, USA. Research supported in part by NSF RTG DMS-1937241 and an AMS-Simons Travel Grant. E-mail: wigal@illinois.edu.}}
\begin{document}
\maketitle

\begin{abstract}
    Let $r \ge 3$ be fixed and $G$ be an $n$-vertex graph. A long-standing conjecture of Gy{\H o}ri states that if $e(G) = t_{r-1}(n) + k$, where $t_{r-1}(n)$ denotes the number of edges of the Tur{\'a}n graph on $n$ vertices and $r - 1$ parts, then $G$ has at least $(2 - o(1))k/r$ edge disjoint $r$-cliques. We prove this conjecture.\\

\end{abstract}

\section{Introduction}

Denote by $t_r(n)$ the number of edges of the Tur{\'a}n graph on $n$ vertices and $r$ parts, which is the  complete $r$-partite graph with class sizes either $\lfloor n/r \rfloor$ or $\lceil n/r\rceil$. 
A classical theorem of Erd{\H o}s, Goodman and P\'osa \cite{ErdosGoodmanPosa1966} states that the edge set of a graph $G$ can be covered by at most  $t_2(n)$  edges and triangles.  Bollob{\'a}s \cite{Bollobas1976} later generalized this result  to  $r$-cliques.

\begin{theorem}[$r = 3$ \cite{ErdosGoodmanPosa1966}, $r \ge 4$ \cite{Bollobas1976}]\label{thm:cover}
    Let $G$ be a graph and let $r \ge 3$. Then there exists a covering of  the edge set of $G$, consisting of $r$-cliques and edges, of size at most $t_{r-1}(n)$.
\end{theorem}

Many generalizations of Theorem~\ref{thm:cover} have been investigated. For example, one can consider a set of cliques covering  the edges of a graph such that every edge lies in exactly one clique of the cover. Such a cover is called as {\it decomposition} of the edge set. Erd{\H o}s, Goodman, and P\'osa \cite{ErdosGoodmanPosa1966} proved for $r = 3$ that  the edge set of an $n$-vertex graph $G$ can be decomposed into at most $t_2(n)$ edges and triangles. Another possible strengthening is to assign non-uniform `costs' to cliques contained in the decomposition.  Denote $\pi(G)$  the smallest cost for a decomposition of the edge set of a graph $G$ into cliques, where the cost of an $r$-clique is $r$ for every $r$. The following  was conjectured by Katona and Tarj{\'a}n~\cite{kattar}, and proved independently by Gy{\H o}ri and Kostochka \cite{GyoriKostochka1979}, Chung \cite{Chung1981}, and Kahn \cite{Kahn1981}.

\begin{theorem}[Gy{\H o}ri and Kostochka \cite{GyoriKostochka1979}, Chung \cite{Chung1981},  Kahn \cite{Kahn1981}]\label{thm:clique_decomposition_size_cost}
    Let $G$ be an $n$-vertex graph, then
    \[ \pi(G) \le 2\cdot t_2(G).\]
\end{theorem}
A closely related question of Erd{\H o}s (see \cite[Problem 43]{Tuza2001} or \cite{Gyori1992}) is as follows: if every $r$-clique has cost $r-1$, then can the edge set of an $n$-vertex graph be decomposed into cliques with total cost at most $t_2(n)$? This was recently shown to hold asymptotically by He, Krueger, and Nguyen in a joint work with the authors of this paper~\cite{BHKNW2024}. Another possible strengthening of Theorem~\ref{thm:clique_decomposition_size_cost} is to have restrictions on the possible supports of a decomposition. For $r \ge 3$ denote $\pi_r(G)$  the smallest cost of a decomposition of the edge set of a graph $G$ into $2$-cliques and $r$-cliques, where the cost of a clique is its number of vertices. Recall, as $\pi(G)$ denotes the smallest cost of decomposing edges of a graph $G$ into cliques, clearly, $\pi_r(G) \ge \pi(G)$. Gy{\H o}ri and Tuza \cite{GyoriTuza1987} proved the following. 

\begin{theorem}[Gy{\H o}ri and Tuza \cite{GyoriTuza1987}]\label{thm:gyori_tuza}
    Let $G$ be an $n$-vertex graph and $r \ge 4$. Then 
    \begin{equation}\label{gyorituza} \pi_r(G) \le 2 \cdot t_{r-1}(G).\end{equation}
\end{theorem}

We let $\nu_r(G)$ denote the maximum number of edge disjoint $r$-cliques in $G$. It is straightforward to verify that \eqref{gyorituza} is equivalent to the following. Let $G$ be an $n$-vertex graph, let $r \ge 4$, and  $k \ge 0$ such that $e(G) = t_{r-1}(G) + k$. Then~\eqref{gyorituza} holds if and only if \[\nu_r(G) \ge \frac{2k}{r(r-2)}.\]
Motivated by this, Gy{\H o}ri \cite{Gyori1991}, see also \cite{Gyori1992,Tuza2001}, conjectured that Theorem~\ref{thm:gyori_tuza} holds asymptotically for $r = 3$, and further, Theorem~\ref{thm:gyori_tuza} can be strengthened for $r \ge 4$. 

\begin{conjecture}[Gy{\H o}ri \cite{Gyori1991,Gyori1992}, Tuza \cite{Tuza2001}] \label{conj:Gyori}
    Let $G$ be an $n$-vertex graph, let $r \ge 3$ be fixed, and let $k \in \mathbb{R}$ such that $e(G)= t_{r-1}(n) +k$. Then 
    \[\nu_r(G) \ge  (2-o(1))k/r.\]
\end{conjecture}

Kr{\'a}l', Lidick{\'y}, Martins, and Pehova \cite{KLMP2019} proved Conjecture~\ref{conj:Gyori} in the special case of $r = 3$. 

\begin{theorem}[Kr{\'a}l', Lidick{\'y}, Martins, Pehova \cite{KLMP2019}]\label{thm:Gyori_3}
    Let $G$ be an $n$-vertex graph, then
    \[ \pi_3(G) \le (1/2 + o(1))n^2.\]
\end{theorem}
Computer assisted flag algebra calculations were essential in the proof of Theorem~\ref{thm:Gyori_3}. Building on Theorem~\ref{thm:Gyori_3} and using stability arguments, Blumenthal,  Lidick\'y,  Pehova,  Pfender,  Pikhurko, and  Volec~\cite{BLPPPV2021} showed $\pi_3(G) \le n^2/2 + 1$ for sufficiently large $n$ and characterized the extremal graphs.
Their work still depended essentially on flag-algebra calculations. Also, using stability arguments, Gy{\H o}ri  \cite{Gyori1991} showed that Conjecture~\ref{conj:Gyori} is true  for small  $k$.

\begin{theorem}[Gy{\H o}ri  \cite{Gyori1991}]\label{thm:stability}
      Let $G$ be an $n$-vertex graph, let $r \ge 3$ be fixed, and let $k(n) = k \in \mathbb{R}$ such that $e(G)=t_{r-1}(n)+k$ and $k = o(n^2)$. Then 
      \[\nu_r(G) \ge k - O\left(\frac{k^2}{n^2}\right) =  (1-o(1))k.\]
\end{theorem}

Before stating our result, we need to introduce some notation. For a graph $G$, write $\mathcal{K}(G)$ for the set of cliques of $G$, and  $\mathcal{K}_r(G)$ for the set of $r$-cliques, where $r\ge 2$. A fractional $r$-clique packing is a mapping $f : \mathcal{K}_r(G) \to \mathbb{R}_{\ge 0}$ such that 
\[ \sum_{\substack{K \in \mathcal{K}(G)\\ e \in E(K)}} f(K) \le 1 \quad \text{ for every } e \in E(G).\]
For   a fractional $r$-clique packing  $f$ we define
\[ \norm{f} = \sum_{K \in \mathcal{K}(G)} f(K).\]
We let $\nu_r^*(G) = \max \norm{f}$, where the maximum is taken over all fractional $r$-clique packings. Standard linear programming arguments show that $\nu_r^*(G)$ is well-defined. Clearly, $\nu_r^*(G) \ge \nu_r(G)$. A well-known result of Haxell and R{\"o}dl \cite{HaxellRodl2001} states that the $r$-clique packing number and its fractional analog are close in dense graphs, see also \cite{Yuster2005}.

\begin{theorem}[Haxell, R{\"o}dl \cite{HaxellRodl2001}]\label{thm:haxell_rodl}
    For all $\varepsilon > 0$ and  $n$ sufficiently large,   for all $n$-vertex graphs $G$,
    \[ \nu^*_r(G) - \nu_r(G) \le \varepsilon n^2.\]
\end{theorem}

Our main result is a continuous analog of Conjecture~\ref{conj:Gyori}. 

\begin{theorem}\label{thm:main}
    Let $G$ be an $n$-vertex graph, let $r \ge 3$ and  $k \in \mathbb{R}$ such that $e(G) = \left(1 - \frac{1}{r-1}\right)\frac{n^2}{2} + k$. Then 
    \[\nu_r^*(G) \ge 2k/r.\] 
\end{theorem}

Let us derive Conjecture~\ref{conj:Gyori} from Theorem~\ref{thm:main}. As a byproduct   of our proof of  Conjecture~\ref{conj:Gyori}, we also obtain flag-algebra free proof of Theorem~\ref{thm:Gyori_3}.

\begin{proof}[Proof of Conjecture~\ref{conj:Gyori} assuming Theorem~\ref{thm:main}]
    By Theorem~\ref{thm:haxell_rodl}, there exists a function $f : \mathbb{N} \to \mathbb{R}$ such that 
    \[ \nu^*_r(G) - \nu_r(G) \le f(n),\]
    where $G$ is an $n$-vertex graph and $f(n) = o(n^2)$. Let $g(n) : \mathbb{N} \to \mathbb{R}$ such that
    \[  \left(1 - \frac{1}{r-1}\right)n^2/2 - t_{r-1}(n) = g(n) = o(n^2).\]
     Finally, let $h : \mathbb{N} \to \mathbb{R}$ be a function such that $f(n) = o(h(n))$, $g(n) = o(h(n))$, and $h(n) = o(n^2)$. Let $G$ be an $n$-vertex graph such that
    \[e(G) = t_{r-1}(n) + k.\]  
    If $k \le h(n)$, then by Theorem~\ref{thm:stability}, we have that $\nu_r(G) \ge (2 - o(1))k/r$. Otherwise, let $k' \in \mathbb{R}$ such that $e(G) = \left(1 - \frac{1}{r-1}\right)n^2/2 + k'$. Note $k = k' + g(n)$. By Theorem~\ref{thm:main}, 
    \[ \nu_r(G) \ge \nu^*_r(G) - f(n) \ge  2k'/r - f(n) = 2(k - g(n))/r - f(n) =  (2 - o(1))k/r.\]
\end{proof}

 The proof of Theorem~\ref{thm:main} follows the general framework set up in \cite{BHKNW2024}. However, there are several differences, we have not found a simple way of `uniting' the two proofs. In Section~\ref{sec:symmetrization}, we prove a specialized fractional analog of `Zykov Symmetrization', effectively reducing the problem to complete multipartite graphs. In Section~\ref{sec:proof}, we prove Theorem~\ref{thm:main} by optimizing over complete multipartite graphs. In Section~\ref{sec:conclusion}, we discuss possible improvements to Theorem~\ref{thm:main}.

\section{Symmetrization}\label{sec:symmetrization}

Let $G$ be a graph. Two vertices $u,v \in V(G)$ are \emph{clones} if $uv \not \in E(G)$ and $N(u) = N(v)$. Let $V_0$ and $V_1$ be two sets of clones of $G$ such that there is no edge between $V_0$ and $V_1$. We shall use the graph operation of `replacing' the neighborhoods of vertices of $V_0$ with the neighborhood of a vertex from $V_1$. With this in mind, let $G[V_0 \to V_1]$ denote the graph where $V(G[V_0 \to V_1]) = V(G)$ and 
\[ E(G[V_0 \to V_1]) = \{ xy \in E(G) : x,y \not\in V_0 \cup V_1 \} \cup \{ xy : x \in V_0 \cup V_1, y \not\in V_0 \cup V_1, zy \in E(G) \},\]
where $z$ is some arbitrary vertex belonging to $V_1$. For a  fixed $c \in \mathbb{R}$ and an integer $r \ge 3$, 
\begin{equation}\label{eq:g_def}
    h_{r,c}(G) := \nu^*_r(G) + c \cdot e(G).
\end{equation}

\begin{lemma}\label{lem:symmetrization}
    Let $G$ be a graph, let $c \in \mathbb{R}$ and $r \ge 3$ be fixed, and let $h_{r,c}$ be defined as in \eqref{eq:g_def}. Let $V_0$ and $V_1$ be sets of pairwise clones of $G$ such that $V_0 \cup V_1$ forms an independent set in $G$. Then,
    \[ h_{r,c}(G) \ge \min \{ h_{r,c}(G[V_0 \to V_1]), h_{r,c}(G[V_1 \to V_0])\}.\]
\end{lemma}

\begin{proof}
    For $i \in \{0,1\}$, let $f_i$ denote an optimal fractional $r$-clique packing of $G[V_{1-i} \to V_i]$. Note that
    \[ f_i' = \frac{1}{|Aut(G)|} \sum_{\pi \in Aut(G)} f_i \circ \pi,\]
    is also an optimal fractional $r$-clique packing of $G[V_{1-i} \to V_i]$ such that $\norm{f_i'} = \norm{f_i}$, where $Aut(G)$ denotes the automorphism group of $G$. Without loss of generality, we may suppose $f_i(Q \cup \{v\})$ is the same for every $v \in V_0 \cup V_1$, where $Q$ is an arbitrary clique in the neighborhood of $V_i$ for $i \in \{0,1\}$. Let $f : \mathcal{K}(G) \to \mathbb{R}_{\ge 0}$ be defined as follows, 
    \begin{align*}
        f(K) = \begin{cases}
        \frac{|V_0|}{|V_0| + |V_1|} \cdot f_0(K) + \frac{|V_1|}{|V_0| + |V_1|} \cdot f_1(K), & \text{when}\quad  |K \cap (V_0 \cup V_1)| = 0,\\
        f_i(K), & \text{when}\quad  |K \cap V_i| = 1 \quad\text{and}\quad i \in \{0,1\}.\\\end{cases}
\end{align*}

We first show that $f$ is an $r$-clique packing. First observe that  for $i \in \{0,1\}$, if $e$ is an edge with an endpoint in $V_i$, as $V_0 \cup V_1$ is an independent set in both $G$ and $G[V_{1-i} \to V_i]$, it follows,
\[ \sum_{\substack{K \in \mathcal{K}(G)\\ e \in K}} f(K) = \sum_{\substack{K \in \mathcal{K}(G) \\ e \in K}} f_i(K)  = \sum_{\substack{K \in \mathcal{K}(G[V_{1-i} \to V_i])\\  e \in K}} f_i(K) \le 1.\]
Now suppose $e$ has no endpoint in $V_0 \cup V_1$. First observe that for $i \in \{0,1\}$, \\

\[ \sum_{\substack{K \in \mathcal{K}(G)\\ K \cap V_i \neq \emptyset\\ e \in K }} f(K) = \sum_{\substack{K \in \mathcal{K}(G)\\ K \cap V_i \neq \emptyset\\ e \in K }} f_i(K) = \frac{|V_i|}{|V_0| + |V_1|} \sum_{\substack{ K \in \mathcal{K}(G[V_{1-i} \to V_i])\\ K \cap (V_0 \cup V_1) \neq \emptyset\\ e \in K }} f_i(K). \] 
It follows, 
\begin{align*}
    \sum_{\substack{K \in \mathcal{K}(G)\\ e \in K}} f(K) &= \sum_{i \in \{0,1\}} \Bigg( \sum_{\substack{ K \in \mathcal{K}(G)\\ K \cap (V_0 \cup V_1) = \emptyset\\ e \in K}} \frac{|V_i|}{|V_0| + |V_1|} f_i(K) +  \sum_{\substack{K \in \mathcal{K}(G)\\ K \cap V_i \neq \emptyset\\ e \in K }} f_i(K)  \Bigg)\\
    &= \sum_{i \in \{0,1\}} \Bigg( \frac{|V_i|}{|V_0| + |V_1|}  \sum_{\substack{ K \in \mathcal{K}(G[V_{1-i} \to V_i])\\ K \cap (V_0 \cup V_1) =  \emptyset\\ e \in K}} f_i(K) +  \sum_{\substack{K \in \mathcal{K}(G)\\ K \cap V_i \neq \emptyset\\ e \in K }} f_i(K)  \Bigg)\\
    &=  \sum_{i \in \{0,1\}} \Bigg( \frac{|V_i|}{|V_0| + |V_1|}  \sum_{\substack{ K \in \mathcal{K}(G[V_{1-i} \to V_i])\\ K \cap (V_0 \cup V_1) = \emptyset\\ e \in K}} f_i(K) +  \frac{|V_i|}{|V_0| + |V_1|} \sum_{\substack{ K \in\mathcal{K}(G[V_{1-i} \to V_i])\\ K \cap (V_0 \cup V_1) \neq \emptyset\\ e \in K }} f_i(K)\Bigg)\\
    &= \sum_{i \in \{0,1\}} \frac{|V_i|}{|V_0| + |V_1|} \Bigg(\sum_{\substack{ K \in \mathcal{K}(G[V_{1-i} \to V_i])\\  e \in K}} f_i(K)\Bigg) \le \sum_{i \in \{0,1\}} \frac{|V_i|}{|V_0| + |V_1|} = 1.
\end{align*}
Thus $f$ is indeed an $r$-clique packing. We now perform a similar calculation for $\norm{f}$. 
\begin{align*}
    \norm{f} &= \sum_{K \in \mathcal{K}(G)} f(K) = \sum_{i \in \{0,1\}} \Bigg( \sum_{\substack{ K \in \mathcal{K}(G)\\ K \cap (V_0 \cup V_1) = \emptyset}} \frac{|V_i|}{|V_0| + |V_1|} f_i(K) +  \sum_{\substack{K \in \mathcal{K}(G)\\ K \cap V_i \neq \emptyset}} f_i(K)  \Bigg)\\
    &= \sum_{i \in \{0,1\}} \Bigg( \frac{|V_i|}{|V_0| + |V_1|}  \sum_{\substack{ K \in \mathcal{K}(G[V_{1-i} \to V_i])\\ K \cap (V_0 \cup V_1) = \emptyset}} f_i(K) 
     +  \frac{|V_i|}{|V_0| + |V_1|}  \sum_{\substack{K \in \mathcal{K}(G[V_{1-i} \to V_i])\\ K \cap (V_0 \cup V_1) \neq \emptyset}} f_i(K)  \Bigg)\\
    &= \sum_{i \in \{0,1\}} \frac{|V_i|}{|V_0| + |V_1|} \Bigg( \sum_{K \in \mathcal{K}(G[V_{1-i} \to V_i])} f_i(K) \Bigg) = \sum_{i \in \{0,1\}} \frac{|V_i|}{|V_0| + |V_1|} \norm{f_i}. 
\end{align*}
We now claim  that 
\begin{equation}\label{eq:edget}
    e(G) = \frac{|V_0|}{|V_0| + |V_1|}e(G[V_1 \to V_0]) + \frac{|V_1|}{|V_0| + |V_1|}e(G[V_0 \to V_1]). 
\end{equation}
Let $u \in V_0$ and $v \in V_1$ be arbitrary, and let $d_G(u) = d_0$ and $d_G(v) = d_1$. Then $e(G[V_1 \to V_0]) = e(G) + |V_1| (d_0 - d_1)$ and $e(G[V_0 \to V_1]) = e(G) + |V_0|(d_1 - d_0)$, implying \eqref{eq:edget}. It follows that
\begin{align*}
    h_{r,c}(G) \ge   \norm{f} &= \frac{|V_0|}{|V_0| + |V_1|} \cdot \norm{f_0}+ \frac{|V_1|}{|V_0| + |V_1|} \cdot \norm{f_1}\\
    &=\frac{|V_0|}{|V_0| + |V_1|} \cdot h_{r,c}(G[V_0 \to V_1]) + \frac{|V_1|}{|V_0| + |V_1|} \cdot  h_{r,c}(G[V_1 \to V_0])\\
    &\ge \min \{ h_{r,c}(G[V_0 \to V_1]), h_{r,c}(G[V_1 \to V_0])\}.
\end{align*}
\end{proof}

We remark that Lemma~\ref{lem:symmetrization} holds for a larger class of functions of $h_{r,c}$ which may be of general interest. The proof remains valid if we replace $h_{r,c}$ with
\[ \nu_r^*(G) + g\left(|\mathcal{K}_2(G)|, \ldots, |\mathcal{K}_b(G)| \right),\]
where $g$ is a concave function and $b \ge 2$ is an arbitrary integer. In \cite{BHKNW2024}, an analog of Lemma~\ref{lem:symmetrization} was proved for weighted covers/decompositions and $g := 0$. A similar statement for a broader family of functions of $g$ can be proved in these settings as well. As this is not pertinent to our proof of Theorem~\ref{thm:main}, we omit these details. Iterating Lemma~\ref{lem:symmetrization} yields the following.

\begin{lemma}\label{lem:multipartite}
    Let $r \ge 3$. For every graph $G$  there exists a complete multipartite graph $H$ such that
    \[ \nu_r^*(G) - \frac{2}{r}e(G)  \ge \nu^*(H) - \frac{2}{r}e(H).\]
\end{lemma}

\begin{proof}
    Let $G$ be a graph and $r \ge 3$. Define an equivalence relation $\sim$ on $V(G)$ where $u \sim v$ if and only if $u$ and $v$ are clones. Let $V_1,\ldots, V_s$ be the equivalence classes under this relation. If for all distinct $i,j \in [s]$, $V_i \cup V_j$ is not an independent set, then $G$ is a complete multipartite graph. Otherwise, there exists distinct $i,j \in [s]$ such that $V_i \cup V_j$ is an independent set. Applying Lemma~\ref{lem:symmetrization} to $V_i$ and $V_j$ with $c = -2/r$ yields a graph $G'$ such that
    \[ \nu_r^*(G) - \frac{2}{r}e(G)  \ge \nu^*(G') - \frac{2}{r}e(G'),\]
    and the number of equivalence classes of $G'$ induced by $\sim$ is strictly smaller than $s$. A straightforward induction on $s$ yields the claim.
\end{proof}

\section{Proof of Theorem~\ref{thm:main}}\label{sec:proof}
Let $G$ be an $n$-vertex graph and $r \ge 3$. It suffices to prove that
\begin{equation}\label{eq:main}
    \nu_r^*(G) -\frac{2}{r}e(G) \ge -\frac{2}{r} \cdot \left(1 - \frac{1}{r-1}\right)\frac{n^2}{2} = -\frac{r-2}{r(r-1)} n^2.
\end{equation}
By Lemma~\ref{lem:multipartite},  we may suppose that $G$ is a complete multipartite graph with $s$ parts, for some  $s$. 
First, in Claim~\ref{clm:trivial_case},  we handle the case, when $e(G)$ is small, which includes the $s\le r-1$ case. In Claim~\ref{clm:trivial_case} we handle the case $s=r$. Observe, that for $r=2$, both sides of \eqref{eq:main} are $0$. Afterwards, we can start our induction on $r+s$, which involves some careful computations.

 \begin{claim}\label{clm:trivial_case}
 Relation~\eqref{eq:main} holds when 
 \[ e(G) \le \left(1 - \frac{1}{r-1}\right)\frac{n^2}{2} = \frac{r-2}{2(r-1)}n^2.\] In particular, \eqref{eq:main} holds when $s \le r - 1$. 
      \end{claim}
\begin{proof}
As $\nu_r^*(G) \ge 0$, Relation~\eqref{eq:main} holds when $e(G) \le \frac{r-2}{2(r-1)}n^2$. Standard optimization techniques yields for  a complete $s$-partite graph $G$, where  $s \le r-1$, that
\[e(G)\le \binom{s}{2}\left(\frac{n}{s}\right)^2 = \left(1 - \frac{1}{s}\right)\frac{n^2}{2} \le \left(1 - \frac{1}{r-1}\right)\frac{n^2}{2}.\]
\end{proof}

As $G$ is complete multipartite, denote by $V_1, \ldots,  V_s$  the parts of $G$, and for all $i \in [s]$, let $x_i := |V_i|/n$. Note $\sum_i x_i = 1$, and we may further suppose $x_1 \ge \ldots \ge x_s \ge 0$.

\begin{claim}\label{clm:rclass}   
          Relation~\eqref{eq:main}  holds when $s = r\ge 3$, i.e., when $G$ is complete $r$-partite.
      \end{claim}
\begin{proof}

Assigning a uniform weighting to $r$-cliques, it is straightforward to verify $\nu^*_r(G) = x_{r-1}x_r$. Note,
\[ \frac{e(G)}{n^2} \le x_{r-1}x_r+ (x_{r-1}+x_r)(1-x_{r-1}-x_r) + \frac{r-3}{2(r-2)}(1-x_{r-1}-x_r)^2.\]
By~\eqref{eq:main} is suffices to show the following: 
\[ x_{r-1}x_r - \frac{2}{r}\left(  x_{r-1}x_r+ (x_{r-1}+x_r)(1-x_{r-1}-x_r) + \frac{r-3}{2(r-2)}(1-x_{r-1}-x_r)^2 \right)
\ge  - \frac{r-2}{r(r-1)}.\]
Multiplying by $r$ and rearranging,
\[ (r-2)x_{r-1}x_r \ge {2}\left(   (x_{r-1}+x_r)(1-x_{r-1}-x_r) + \frac{r-3}{2(r-2)}(1-x_{r-1}-x_r)^2 \right)
 - \frac{r-2}{r-1}. \]
Multiplying with $r-2$ and further algebraic rearranging yields the following. 
\[ (r-2)^2x_{r-1}x_r \ge (2r-4) (x_{r-1}+x_r)(1-x_{r-1}-x_r) + (r-3)(1-x_{r-1}-x_r)^2 
 - \frac{(r-2)^2}{r-1}. \]
Further rearranging results in the following,
\begin{align*}
   \frac{r^2 - 4r + 4}{r-1} + (r^2 - 4r + 4)x_{r-1}x_r &\ge (1-x_{r-1}-x_r)[(2r-4)(x_{r-1}+x_r) + (r-3)(1 - x_{r-1}-x_r)] \\
   &= (1-x_{r-1}-x_r) [r-3+ (r-1) x_{r-1}+(r-1) x_r]\\
   &= r - 3 + 2x_{r-1} + 2x_r - 2(r-1)x_{r-1}x_r - (r-1)x_{r-1}^2 - (r-1)x_r^2.
\end{align*}
Rearranging more, we have that the above inequality is equivalent to the following.
\begin{align*}
& \frac{1}{r-1} + (r^2-2r+2)x_{r-1}x_r +(r-1)x_{r-1}^2 +(r-1)x_{r}^2 - 2x_{r-1} - 2x_r\\
=\  & (r-1)\left(x_{r-1}+ x_r-\frac{1}{r-1}\right)^2 +(r-2)^2x_{r-1}x_r
\ \ge\  0.
\end{align*}
\end{proof}

 We now proceed with induction on  $r+s$. As a base case, we already settled the case when $r=2$, and by Claims~\ref{clm:trivial_case}~and~\ref{clm:rclass}, we already handled the cases when  $s \le r$. Hence, we will assume that $s>r\ge 3$, and we know that the induction hypothesis holds for the pairs $(r-1,s-1)$ and $(r,s-1)$.\\

 Let $k_G \in \mathbb{R}$ such that $k_G = e(G) - t_{r-1}(G)$, i.e.,
\begin{equation}\label{eq:e(G)}
    e(G) = \left(1 - \frac{1}{r-1}\right)\frac{n^2}{2} + k_G  = \frac{r-2}{2(r-1)}n^2 + k_G. 
\end{equation}
 By Claim~\ref{clm:trivial_case}, we may suppose $e(G) > \frac{r-2}{2(r-1)}n^2$, which implies $k_G > 0$. It suffices to show that 
$\nu_r^*(G) \ge 2k_G/r$. Define $H := G - V_1$. We shall need to compute $e(G)$ in two additional ways. Let $t_H \in \mathbb{R}$ such that $t_H = e(H) - t_{r-2}((1 - x_1)n)$, then,
\begin{align*}
     e(G) &= e(H) + x_1(1-x_1)n^2= 
\frac{1}{2} \left( 1 - \frac{1}{r-2}\right) (1-x_1)^2n^2 +x_1(1-x_1)n^2+ t_H\\
&= \frac{1}{2} \left( 1 - \frac{1}{r-2}\right) n^2 -
x_1 \left( 1 - \frac{1}{r-2}\right) n^2
+ \frac{x_1^2}{2} \left( 1 - \frac{1}{r-2}\right) n^2
+x_1(1-x_1)n^2+ t_H\\
&=  \frac{1}{2} \left( 1 - \frac{1}{r-2}\right) n^2 + \frac{x_1}{r-2} n^2 - \frac{x_1^2(r-1)}{2(r-2
)}n^2 + t_H.
\end{align*}

 By \eqref{eq:e(G)}, 
\begin{equation}\label{eq:t_H}
    t_H = k_G  +  \frac{n^2}{2(r-1)(r-2)}   
- \frac{x_1}{r-2} n^2 + \frac{x_1^2(r-1)}{2(r-2)}n^2.
\end{equation}

Now let $k_H \in \mathbb{R}$ such that $k_H = e(H) - t_{r-1}((1 - x_1)n)$, then,
\begin{align*}
    e(G) &= e(H) + x_1(1-x_1)n^2= 
\frac{1}{2} \left( 1 - \frac{1}{r-1}\right) (1-x_1)^2n^2 +x_1(1-x_1)n^2+ k_H\\
&=  \frac{1}{2} \left( 1 - \frac{1}{r-1}\right) n^2 -
x_1 \left( 1 - \frac{1}{r-1}\right) n^2
+ \frac{x_1^2}{2} \left( 1 - \frac{1}{r-1}\right) n^2
+x_1(1-x_1)n^2+ k_H\\
&= \frac{1}{2} \left( 1 - \frac{1}{r-1}\right) n^2 + \frac{x_1}{r-1} n^2 - \frac{rx_1^2}{2(r-1)}n^2 + k_H.
\end{align*}

In particular, by \eqref{eq:e(G)},
\begin{equation}\label{eq:k_H}
    k_H = k_G - \frac{x_1}{r-1}n^2 +\frac{rx_1^2}{2(r-1)}n^2.
\end{equation}

\begin{claim}\label{clm:t_H_bound}
    We have that $k_G \le t_H$. In particular, $t_H \le 0 $ implies $k_G \le 0$.
\end{claim}

\begin{proof}
    By \eqref{eq:t_H}, it suffices to prove the following inequality.    \\

 \[      \frac{n^2}{2(r-1)(r-2)}   + \frac{x_1^2(r-1)}{2(r-2)} n^2 \ge \frac{x_1}{r-2} n^2.\] 
Multiplying by $2(r-2)/n^2$, it  can be rewritten as 
 \[ (r-1)\left(x_1-\frac{1}{r-1}\right)^2\ge 0. \]
\end{proof}

Define
\begin{equation}\label{eq:alpha}
    \alpha := \min \left\{ \frac{1}{n x_1}, \frac{r-2}{n(1-x_1)}\right\}.
\end{equation}

\begin{claim}\label{clm:H_i_greedy}
   There exists an $r$-clique packing $f'$ of $G$  such that for every $uv \in E(H)$ we have
    \[\sum_{\substack{ K \in \mathcal{K}(H)\\ uv \in E(K)}} f'(K) \le  nx_1 \alpha \quad\quad\text{ and }\quad\quad \norm{f'} \ge nx_1\alpha\frac{2t_H}{r-1}.\]
\end{claim}

\begin{proof} 
     If $s>r > 3$, then by induction on $s$ and $r$, there is an   optimal $(r - 1)$-clique packing $h$ 
of $H$ such that $\norm{h} \ge \frac{2}{r-1}t_H$. Otherwise, if $r = 3$, then we let $h$ denote the identity map of $E(G)$ (which is essentially the same as reducing to the $r-1=2$ case),  and trivially  $\norm{h} = e(H) \ge \frac{2}{r-1}t_H$. We will extend $h$ to a fractional $r$-clique packing in $G$, by adding vertices from $V_1$ to each $(r-1)$-clique in the support of $h$. 
    Define the map $f' : \mathcal{K}_r(G) \to \mathbb{R}_{\ge 0}$ where
    \[ f'(K) = \begin{cases} \alpha \cdot  h(Q) &\text{ if } K \cap V_1 \neq \emptyset\text{ and } Q\subset K;\\
    0 &\text{ otherwise}.\end{cases}\]
     We first show that $f'$  satisfies the packing constraints. We break the proof into two cases. Let $uv \in E(G)$ where  $v \in V_1$ and $u \not \in V_1$. As $u$ has degree at most $(1-x_1)n$ in $H$ and an  $(r-1)$-clique is an  $(r-2)$-regular graph, we have 
    \[ \sum_{\substack{ Q \in \mathcal{K}_{r-1}(H)\\ u \in Q}} h(Q) \le \frac{(1-x_1)n}{r-2}.\]
    By the definition of $\alpha$, see \eqref{eq:alpha}, we conclude
    \[ \alpha  \cdot \hspace{-1em} \sum_{\substack{K \in \mathcal{K}_r(G) \\ uv \in E(K)}} f'(K) \le 1.\]
    Now assume $uv \in E(H)$. In this case we have
    \[ \sum_{\substack{ K \in \mathcal{K}_r(G)\\ uv \in E(K)}} f'(K)  = nx_1\alpha \cdot \hspace{-1em} \sum_{\substack{ Q \in \mathcal{K}_{r-1}(H)\\ uv \in E(Q)}} h(Q)   \le nx_1 \alpha \le 1,\]
    where the last inequality follows from  \eqref{eq:alpha}.
    We  conclude $f'$ is indeed a packing. 
    Finally, we have,
    \[ \norm{f'} = nx_1\alpha \cdot  \norm{h} \ge  nx_1\alpha\frac{2t_H}{r-1}.\]  
\end{proof}

If  $\norm{f'}  \ge \frac{2k_G}{r}$ then we are done. For the rest of the proof we assume that it does not hold. As an immediate consequence of Claim~\ref{clm:H_i_greedy}, we can determine $\alpha$. 

\begin{claim}\label{clm:alpha}
  If  $\norm{f'}  < \frac{2k_G}{r}$ then  
    \[ \alpha = \frac{r-2}{n(1-x_1)}.\]
\end{claim}

\begin{proof}
    Suppose otherwise, then by \eqref{eq:alpha} we have that $\alpha = 1/(nx_1)$. Let $f'$ be the $r$-clique packing of $G$ as guaranteed by Claim~\ref{clm:H_i_greedy}.  By Claims~\ref{clm:t_H_bound}~and~\ref{clm:H_i_greedy}, we have the following contradiction: 
    \[ \norm{f'} \ge nx_1 \alpha \frac{2 t_H}{r-1} \ge \frac{2k_G}{r}.\]
\end{proof}

\begin{claim}\label{clm:k_ub}
    The following holds:
    \begin{align*} 
    k_G &< \frac{n^2}{2(r-1)} - x_1\frac{n^2}{2}.\\
    \end{align*}
\end{claim}

\begin{proof}

First suppose $k_H > 0$. Let $g$ 
be an optimal $r$-clique packing in $H$.  By induction on $s$,  $\norm{g} \ge 2 k_H/r$.  
Let $f'$ be an $r$-clique packing of $G$ as guaranteed by Claim~\ref{clm:H_i_greedy}.  By Claim~\ref{clm:alpha}, $\alpha =  \frac{r-2}{n(1-x_1)}$ and  by Claim~\ref{clm:H_i_greedy}, for all $uv \in E(H)$, 
\[ \sum_{\substack{K \in \mathcal{K}(G) \\ uv \in E(K)}} f'(K) \le nx_1 \alpha = \frac{x_1(r-2)}{1-x_1}, \]
and
\[ \norm{f'} \ge nx_1 \alpha \frac{2t_H}{r-1} = \frac{2x_1(r-2)}{(1 - x_1)(r-1)} t_H.\]
Then,
\[f := f' + \left(1 - \frac{x_1(r-2)}{(1-x_1)} \right) g,\]
is also an $r$-clique packing of $G$. Furthermore, 
\begin{equation}\label{eq3} \norm{f} \ge \frac{2x_1(r-2)}{(1-x_1)(r-1)}t_H + \frac{2}{r} \left( 1 -   \frac{x_1(r-2)}{(1-x_1)}\right) k_H.
\end{equation}
Observe, that \eqref{eq3} holds when $k_H \le 0$, as in the above proof we could just define $g$ to be the zero map, i.e. $f=f'$. To show that $f$ is our desired $r$-clique packing of $G$, it is enough to show 
\begin{equation}\label{eq4} \frac{2x_1(r-2)}{(1-x_1)(r-1)}t_H + \frac{2}{r}\left( 1 -   \frac{x_1(r-2)}{(1-x_1)}\right)k_H \ge \frac{2}{r} k_G.\end{equation}
To conclude the proof of the claim, we show \eqref{eq4} when $k_G \ge n^2/(2(r-1)) - x_1n^2/2$. Multiplying \eqref{eq4} by $r(1-x_1)(r-1)/2n^2$ yields
\begin{equation}\label{eq:k_ub_eq}
   x_1r(r-2)\frac {t_H}{n^2} + (r-1)(1 - x_1(r-1))\frac{k_H}{n^2} \ge (1-x_1)(r-1)\frac{k_G}{n^2}. 
\end{equation}
Using \eqref{eq:t_H}
and  \eqref{eq:k_H},
the left-hand side of the inequality~\eqref{eq:k_ub_eq} is equal to the following.
\[ x_1r(r-2)\frac{k_G}{n^2} + \frac{rx_1}{2(r-1)} - rx_1^2 + \frac{x_1^3r(r-1)}{2} + (r-1)(1 - x_1(r-1))\frac{k_G}{n^2} - (1 - x_1(r-1))x_1  +  \frac{rx_1^2(1-x_1(r-1))}{2}.\]
We can subtract a $(1-x_1)(r-1)k_G/n^2$ term from both sides of the inequality, thus inequality~\eqref{eq:k_ub_eq} is equivalent to showing the following is nonnegative,
\begin{align*}
     x_1r(r-2)\frac{k_G}{n^2} + \frac{rx_1}{2(r-1)} - rx_1^2 + \frac{x_1^3r(r-1)}{2} - (r-1)(r-2)x_1\frac{k_G}{n^2} - (1 - x_1(r-1))x_1  +  \frac{rx_1^2(1-x_1(r-1))}{2}.
\end{align*}
Multiplying by $2/x_1$ yields,
\begin{align*}
     2r(r-2)\frac{k_G}{n^2} + \frac{r}{r-1} - 2rx_1 + x_1^2r(r-1) - 2(r-1)(r-2)\frac{k_G}{n^2} - 2(1 - x_1(r-1))  + rx_1(1-x_1(r-1)) \ge 0.
\end{align*}
The left hand side after simplification is 
\begin{align*}
       &= 2(r-2) \frac{k_G}{n^2} - \frac{r-2}{r-1}  - 2rx_1 + x_1^2r(r-1) + 2x_1(r-1)  + rx_1(1-x_1(r-1))\\
       &= 2(r-2) \frac{k_G}{n^2} - \frac{r-2}{r-1} +(r-2)x_1  + x_1^2r(r-1) -x_1^2r(r-1)=  2(r-2) \frac{k_G}{n^2} - \frac{r-2}{r-1} +(r-2)x_1.
\end{align*}
Thus $f$ is the desired packing if the final function is non-negative, i.e.,
\[ k_G \ge \frac{n^2}{2(r-1)} - x_1\frac{n^2}{2}.\] 
\end{proof}

\begin{claim}\label{clm:e(G)_lb}
    We have that $e(G) \ge (1-x_1)n^2/2$. 
\end{claim}

\begin{proof}
    First observe for arbitrary $1 \le  i < j \le s$, if $G'$ is a complete multipartite graph with parts $V_1',\ldots,V_s'$ such that for every $k \in [s]$, 
    \[ |V_k'| = \begin{cases}
        |V_i| + 1 \quad&\text{if}\quad k = i,\\
        |V_j| - 1 \quad&\text{if}\quad k = j,\\
        |V_k| \quad&\text{otherwise,}
    \end{cases}\]
    then $e(G') \le e(G)$. Applying this operation, we shall arrive to a  $G'$,  where $x_i' = x_1$ for $i \in \{1,\ldots,\floor{1/x_1}\}$ and for $i = \floor{1/x_1} + 1$, $x_i' = (1 - \floor{1/x_1}x_1)$ (possibly equal to zero).
    Let $z = 1/x_1 - \floor{1/x_1}$ and note $z \in [0,1)$. 
    Now,
    \begin{align*}
        \frac{e(G)}{n^2} &\ge \frac{e(G')}{n^2} = \binom{\floor{1/x_1}}{2}x_1^2  + (1 - \floor{1/x_1}x_1)x_1\floor{1/x_1}\\
        &= \frac{(1/x_1 - z)(1/x_1 - z - 1)}{2}x_1^2  + (1 - (1/x_1 - z)x_1)x_1(1/x_1 - z)\\
        &=  \frac{(1 - zx_1)(1 - (z+1)x_1)}{2} + (1 - (1 - x_1z))(1 - x_1z)\\
        &= \frac{1 - (2z + 1)x_1 + (z^2 + z)x_1^2}{2} + x_1z(1-x_1z)
        = \frac{1}{2} - \frac{x_1}{2} + \frac{(z - z^2)x_1^2}{2} \ge \frac{1}{2} - \frac{x_1}{2}.
    \end{align*}    
\end{proof}

By Claim~\ref{clm:e(G)_lb}, $k_G \ge (1 - x_1)n^2/2 - \left(1 - \frac{1}{r-1}\right)n^2/2$. By Claim~\ref{clm:k_ub}, 
    \[ \frac{n^2}{2(r-1)} - x_1\frac{n^2}{2} = (1 - x_1)\frac{n^2}{2} - \left(1 - \frac{1}{r-1}\right)\frac{n^2}{2} \le k_G < \frac{n^2}{2(r-1)} - x_1\frac{n^2}{2},\]
    which is a contradiction, completing the proof of our main result.

\section{Concluding remarks}\label{sec:conclusion}

Define the following function, 
\[ \phi_r(n,k) = \min_G \nu_r(G),\]
where the minimum is taken over all $n$-vertex graphs $G$ such that $e(G) = t_{r-1}(n) + k$. By Tur{\'a}n's Theorem, $\phi_r(n,0) = 0$. By Wilson's Theorem~\cite{Wilson1975}, for the complete graph $K_n$, we have that $\phi_r(n,e(K_n) - t_{r-1}(G)) = (1 - o(1))\binom{n}{2}/\binom{r}{2}= (2 - o(1))(n^2/(2r-2))/r$. By our proof of Conjecture~\ref{conj:Gyori} we have
\[ \phi_r(n,k)  \ge (2 - o(1))k/r. \]
While, our bound is asymptotically sharp  at the Tur\'an graph $T_{r-1}(n)$ and the complete graph, for the other values we do not have a conjecture. A direct consequence of Theorem~\ref{thm:stability} is $\phi_r(n,k)  = (1 - o(1))k$ when $k = o(n^2)$. Gy{\H o}ri \cite{Gyori1991} showed for fixed $r \ge 4$ and for sufficiently large $n$, $\phi_r(n,k)  = k$ when $k \le 3\floor{\frac{n+1}{r-1}} - 5$. 

For $r = 3$, the problem of determining the behavior of $\phi_3(n,k)$ dates back to Erd{\H o}s~\cite{Erdos1971}. Gy{\H o}ri \cite{Gyori1988} proved, see \cite{Gyori1992} for minor correction, $\phi_3(n,k)  = k$ if $k \le 2n - 10$ when $n$ is odd or if $k \le 1.5n-5$ when $n$ is even. A very precise result of  Gy{\H o}ri and Keszegh~\cite{GyoriKeszegh2017} claims that a $K_4$-free graph on $n^2/4+k$ edges contains $k$ pairwise edge-disjoint triangles. Here $k\le n^2/12$.

However, the $K_4$-freeness is important, as the following example shows:
Partition the vertex set of $G$ into three classes $A,B,C$, where $G[A]$ is a complete graph, and  $G[A\cup C, B]$ spans a complete bipartite graph, where $|A|=t, |B|= n/2-t/6, |C|=n/2-5t/6.$
Then $e(G)= t(t-1)/2+ n^2/4-t^2/36$, i.e., $k= 17t^2/36-t/2$. The number of triangles is at most 
\[ |\mathcal{K}_3(G)| \le f(t) := (t^2/2+tn/2- t^2/6)/3,\]
as the edges between $B$ and $C$ are not part of any triangle.
Set $t_{\varepsilon} :=(1+\varepsilon)6n/13$, where $\varepsilon>0$ is an  arbitrary small constant. For sufficiently large $n$, $f(t_\varepsilon)< (1-\varepsilon/100) k$, i.e., the number of triangles drops under $(1-o(1))k$, when $k> 17n^2/169> n^2/12$.

One could easily extend this example for larger $r$. We have not done it as we do not see any reasons why they would be best possible. It seems interesting to determine the range of $k$ when $\phi_r(n,k) = (1-o(1))k$.\\

\noindent
\textbf{Acknowledgments} We would like to thank the anonymous referees for their helpful suggestions.

\end{document}